\documentclass[12pt,reqno]{amsart}

\usepackage{fancyhdr}
\usepackage{psfrag}
\usepackage{graphicx}
\usepackage{amsmath}
\usepackage{amsfonts}
\usepackage{amssymb}		
\usepackage{amsthm}
\usepackage{indentfirst}
\usepackage{float}
\usepackage{latexsym}
\usepackage{graphics}
\usepackage{verbatim}

\newcommand{\Q}{{\mathbb Q}}
\newcommand{\C}{{\mathbb C}}

\newcommand{\Z}{{\mathbb Z}}

\newcommand{\Frob}{{\rm Frob}}

\newtheorem{theorem}{Theorem}

\newtheorem{proposition}{Proposition}

\theoremstyle{remark}
\newtheorem{remark}{Remark}

\begin{document}

\author{Vijay M. Patankar}
\address{School of Physical Sciences, Jawaharlal Nehru University, New
  Delhi, INDIA 110067.\\
email: vijaypatankar@gmail.com}  
\author{C. S. Rajan}
\address{School of Mathematics, Tata Institute of Fundamental
  Research, Dr. Homi Bhabha Road, Colaba, Bombay, INDIA 400005.\\
email: rajan@math.tifr.res.in}


\title{Distinguishing Galois representations by their normalized traces}

\begin{abstract}  

Suppose \( \rho_1 \) and \( \rho_2 \) are two pure
Galois representations of the absolute Galois group of a number field
$K$ of weights \( k_1 \) and \( k_2 \) respectively,  having equal
normalized Frobenius traces \( Tr(\rho_1(\sigma_v)) /Nv^{k_1/2}\) and
\( Tr(\rho_2(\sigma_v)) /Nv^{k_2/2}\) at a set of primes \( v\) of $K$
with positive upper density. Assume further that the algebraic
monodromy group of $\rho_1$ is connected and the repesentation is
absolutely irreducible.  We prove that \( \rho_1 \) and \( \rho_2 \)
are twists of each other by power of a Tate twist times  a character
of finite order.  

We apply this to modular forms and deduce a result proved by Murty and
Pujahari. 

\end{abstract}

\maketitle

\section{Introduction}

Let \( f \) be a Hecke eigenform  of weight \( k \) on \( \Gamma_1 (
N) \).  For a prime \( p \nmid N \),  denote by \( a_p (f ) \) the
eigenvalue of the Hecke operator $T_p$ at \( \ p \). In (\cite{R}),
refinements of the strong multiplicity one theorem for Hecke
eigenforms are given, where one compares the Hecke eigenvalues for two
modular forms $f_1$ and $f_2$. 

In the context of automorphic representations, it is convenient to
work with the normalized Hecke eigenvalue, \( \tilde{a}_p (f) : =
a_p(f) / p^{(k-1)/2} \).  This has the effect that the $L$-function
associated to the `normalized' system of eigenvalues satisfies a
functional equation of the form $s\leftrightarrow 1-s$.   

In a recent article (\cite{MP}), Murty and Pujahari prove a refinement
of the strong multiplicity one theorem for the normalized eigenvalues
associated to Hecke eigenforms:

\begin{theorem}[Murty-Pujahari \cite{MP}] \label{MP}  Suppose \( f_1
\) and \( f_2 \) be two non-CM Hecke eignenforms of weights \( k_1 \)
and \( k_2 \) respectively. Let \( S \) be the set of primes \( p \)
not dividing the levels of $f_1$ and $f_2$,  such that \(\tilde{a}_p
(f_1 ) = \tilde{a}_p (f_2 ) \).  

If the upper density of \( S \) is positive, then \( k_1 = k_2 \), and
\( f_1 \) and \( f_2 \) are twists of each other by a Dirichlet
character. 
\end{theorem}
The proof by Murty and Pujahari uses the work of Harris,
Taylor and others on Sato-Tate conjecture and is analytic in nature.

In this article,  we prove an analogous theorem in the context of pure
Galois representations. This is a variation of an algebraic
Chebotarev theorem proved in \cite{R}.     

Using Galois representations attached to modular forms we deduce the
above mentioned result of Murty and Pujahari. Our methods also yield a
similar result for Hilber modular forms. 

\section{Theorem}

Let $K$ be a global field,  and $G_K:={\rm Gal}(\bar{K}/K)$ denote the
absolute Galois group over $K$ of a separable closure $\bar{K}$ of
$K$. For a finite place $v$ of $K$, let $Nv$ denote the cardinality of
the residue field at $v$, Let \( \Sigma_K \) denote the set of finite
places of \( K \).

Let $\ell$ be a rational prime coprime to the characteristic of \( K
\) and let $F$ be a $\ell$-adic local field of characteristic zero.
Let $\rho: ~G_K\to GL_n(F)$ be  a continuous  representation of
$G_K$. Since we want to recover the representation from a knowledge of
its traces, we will assume throughout that the representations
considered are semisimple. 

We will assume that $\rho$ is unramified outside a finite set of
places of $K$. At a finite place $v $ of $K$ where $\rho$ is
unramified, let $\sigma_v$ denote the Frobenius conjugacy class in the
quotient group $G_K/{\rm Ker}(\rho)$. If $L$ is a finite extension of
$K$, then we denote by $\rho|_L$ the restriction of the representation
$\rho$ to the subgroup $G_L\subset G_K$. 

Let $E$ be a number field. Define  a $\ell$-adic representation $\rho$
to be $E$-valued, if the characteristic polynomials of the Frobenius
conjugacy classes at unramified primes of $\rho$ has coefficients in
$E$.  

Define the algebraic monodromy group \( G \) attached to \( \rho \) to
be the smallest algebraic subgroup \( G \) of \( GL_n \) defined over
\( F \) such that \( \rho(G_K)\subset G(F) \). This is also the
Zariski closure  of \( \rho(G_K) \) in \( G(F) \).

A Dirichlet character, is a character \( \chi \) of \( G_K\) into
$\bar{F}^*$ of finite order. 

Consider the $\ell$-adic Tate module of ${\mathbb G}_m$, defined as
$\projlim_n \mu_{\ell^n}$, where $\mu_{\ell^n}$ is the group of
${\ell^n}$-roots of unity in $\bar{K}^*$.  This is a $G_K$-module and
denote by $\omega$ its character. For a prime $v$ of $K$ with residue
characteristic not equal to $\ell$, we have $\omega(\sigma_v)=Nv$. The
character $\omega$ is the cyclotomic character acting on the roots of
unity, and is also the inverse of the Tate twist. 

\begin{theorem} \label{maintheorem}
Let \( \rho_1 \), \( \rho_2 \) be  $E$-valued \( \ell \)-adic Galois
representations of $G_K$.  Let $k_1$ and $k_2$ be integers.  
Let $S$  \( v \) of \( K \) unramified for $\rho_1$ and $\rho_2$ such that
$ Tr ( \rho_1 (\sigma_v)/Nv^{k_1/2} ) = Tr ( \rho_2 (\sigma_v)/Nv^{k_2/2} ) $.
Suppose that the Zariski closure \( H_1 \) 
of the image \( \rho_1 (G_K) \) in \( GL_r \) is a connected algebraic
group. 

If the upper density of $S$ is positive, then $k_1=k_2+2r$ 
for some integer $r$ and there is a finite extension
$L$ of $K$, such that 
\[ \rho_2|_L \simeq \rho_1|_L \otimes \omega^r .\]
If in addition, $\rho_1$ is absolutely irreducible, 
then there is a Dirichlet character $\chi$ of $G_K$, such that 
\[ \rho_2\simeq \rho_1\otimes \omega^r \chi.\]
\end{theorem}

\section{Proof of Theorem \ref{maintheorem}}

\begin{proposition} Let \( \rho\) be a \( \ell \)-adic Galois
representations of $G_K$. Assume that the algebraic monodromy group
$G$ of $\rho$ is connected. Then the set $V$ of places $v$ of
$\Sigma_K$ which is unramified for $\rho$ and  where \( Tr(
\rho(\sigma_v ) ) = 0 \) is of density zero. 
\end{proposition} 

\begin{proof} Consider the variety 
\[ 
	X=\{g\in GL_n\mid Tr(g)=0.	\}
\]
This is a closed subvariety in $GL_n$ and closed under conjugation. 
By the algebraic Chebotarev theorem \cite[Theorem 3]{R}, the set $V$
has density equal to the ratio of the number of connected components
of $G$ contained inside $X$ to the total number of connected
components of $G$. The connected component of identity of  $G$ cannot
be contained in $X$, since the trace of the identity element is
$n$. Since $G$ is connected, this proves the proposition. 
\end{proof}

We now observe that \( k_1 \) and \( k_2 \) have the same parity.  Let
\( S^\prime \) be the set of places \( v \in S \) such that \( Tr (
\rho_1 ( Frob_v ) ) \neq 0 \) and \( Tr ( \rho_2 ( Frob_v ) ) \neq 0
\). By the above proposition the upper density of \( S^\prime \)
equals the upper density of \( S \) and is positive. 

Let $S'_1$ denote the subset of places of $S'$ which is of degree one
over $\Q$. Since the set of places of a number field $K$ of degree one
over the rationals is of density one in $K$, it follows that the upper
density of  $S'_1$ is positive.  For \( v  \in S^\prime_1 \), 
\[ \frac{Tr ( \rho_1 ( \sigma_v ) )}{ Tr ( \rho_2 (\sigma_v ) )} 
= Nv^{(k_1 - k_2)/2} .\]  
Furthermore, by our assumtion that the representations 
are \( E \)-valued, it follows that 
\( Nv^{(k_1 - k_2)/2} = p^{(k_1 - k_2)/2} \in E \) for a set of primes of
positive density. This forces that \( k_1 \) and \( k_2 \) to have the
same parity. Let \( 2r = k_1 - k_2 \).  

Consider now the two representations, \( \rho_1 \) and 
\( \rho_2 \otimes \omega^r \). By our assumption, for \( v \in S \), 
\begin{equation}\label{eqtraces}
	Tr ( \rho_1 (\sigma_v ) ) = Tr ( \rho_2 \otimes \omega^r (\sigma_v) ).
\end{equation}

We recall the following theorem proved in \cite[Theorem 2, Part (iii)]{R}:

\begin{theorem}  Let \( \rho_1 \), \( \rho_2 \) be \( \ell \)-adic
Galois representations. Let \( SM( \rho_1 , \rho_2 ) \) be the set of
places \( v \) of \( K \) such that \( Tr ( \rho_1 ( \sigma_v ) ) = Tr
( \rho_2 (\sigma_v )  ) \). 

Assume that the Zariski closure \( H_1 \) of the image \( \rho_1 (G_K)
\) in \( GL_r \) is a connected algebraic group, and  the upper
density of \( SM (\rho_1, \rho_2) \) is positive. Then there exists a
finite extension $L$ of $K$ such that \( \rho_2|_L \simeq \rho_1|_L
\). 

If in addition \( \rho_1 \) is absolutely irreducible, then there
exists  a Dirichlet character $\chi$ of $G_K$, such that \( \rho_2
\simeq \rho_1 \otimes \chi \).
\end{theorem}

Since Equation \ref{eqtraces} holds at the set of places $v\in S$
which has positive density, and by our hypothesis,  
Theorem \ref{maintheorem} follows 
from the foregoing theorem. 

\section{Application to modular forms}

We first recall some well known facts about modular forms. 
Let $N, ~k$ be natural numbers, and $f$ a newform of level $N$, weight
$k$ and Nebentypus character $\epsilon$. 
By a theorem of Shimura, it is known that the field $E$ generated by
the Hecke eigenvalues $a_p(f), ~p\nmid N$ is a number field.
By theorems of Eichler, Shimura, Ihara, Deligne and Serre, 
given any rational prime $\ell$, there is a $\ell$-adic local field
$F\supset E$, and  a two dimensional Galois
representation $\rho_f :G_{\Q}\to GL_2(F)$ unramified outside $N$, 
such that for primes
$p\nmid N$, 
\[ Tr(\rho_f(\sigma_p))=a_p(f)\quad \mbox{and} 
\quad det(\rho_f(\sigma_p)) =\epsilon(p)p^{k-1}.
\]
By theorems of Serre and Ribet, it is known that if the form $f$ is
non-CM, then the algebraic monodromy group is isomorphic to $GL_2$. 

\begin{proof}[Proof of Theorem \ref{MP}]
We deduce now Theorem \ref{MP} from Theorem \ref{maintheorem}. 
Under the assumptions of Theorem \ref{MP}, by Theorem
\ref{maintheorem}, we conclude that $k_1=k_2+2r$ for some integer $r$,
and there exists a Dirichlet character $\chi$ of $G_{\Q}$ such that 
\[ 
	\rho_{f_1}=(\rho_{f_2}\otimes \chi) \otimes \omega^r.
\]
By Proposition \ref{mf} proved below, 
 $r=0$, and hence Theorem \ref{MP} follows. 
\end{proof}

Given a Hecke eigenform \( g \in S_k (N) \), 
\( L(g, s) \), the \( L \)-function associated to \( g \) is given by
\[
	L(g, s ) = \sum_{n \geq 1 } \frac{ a_n (g) }{ n^s } .
\]
It has an Euler product expansion given by
\[
	L(g, s ) := \prod_{p \mid N} \left( 1 - \frac{a_p (g)}{ p^s }  
	\right)^{-1}  \prod_{p \nmid N} \left( 1 - \frac{a_p (g)}{ p^s } +  
	\frac{p^{k-1}}{p^{2s}} \right)^{-1}  = \prod_p  E_p ( g, s ) ,
\]
where \( E_p ( g, s) \) is the Euler factor at \( p \) 
depending on \( p \) dividing \( N\) or not. For a natural number $M$ divisible 
by \( N \), define the partial $L$-function, 
\[
	L_M(g,s)=\prod_{p \nmid M }
\left(1 - \frac{a_p (g)}{ p^s }+\frac{p^{k-1}}{p^{2s}} \right)^{-1} . 
\]
The completed \( \Lambda \) function defined by 
\(  \Lambda ( g , s) = N^{s/2} ( 2 \pi )^{-s} \Gamma (s) L ( g, s ) \) 
satisfies the functional equation of the form:
\[
	\Lambda (g, s) = \lambda(g)  \Lambda ( g , k -s )  ,
\]
where \( \lambda(g) \) is a non-zero complex number and \( \Gamma (s ) \) 
is the \( \Gamma \)-function. Thus \( L (g, s ) \) satisfies a
functional equation of the form
\[
	A_g (s ) \Gamma (s ) L (g, s) 
		= \lambda (g) A_g ( k -s ) \Gamma ( k -s ) L (g, k -s) ,
\] 
where 
\( A_g (s ) \) is a very specific entire function of \( s \) 
that is nowhere zero. 

We give a proof of the  following folklore fact (see also (\cite{MM}):  

\begin{proposition}\label{mf}  
Suppose \( f_1 \) and \( f_2 \) be two Hecke
eignenforms of weights \( k_1 \) and \( k_2 \) and levels  \( N_1 \)
and \( N_2 \), respectively. Suppose there exists a non-negative
integer $r$, such that  $k_1=k_2+2r$ and
$\rho_{f_1}=\rho_{f_2}\otimes \omega^r$, where \( \omega \) is a \(
\ell \)-adic Tate character.  Then $r=0$. 
\end{proposition}

\begin{proof} We rephrase our assumptions:  We are given two cusps
forms  \( f_1 \) and \( f_2 \) belonging to \( S_{k_1} ( N_1) \) and
\( S_{k_2} (N_2) \) respectively such that their \( p \)-th  Fourier
coefficients for primes \( p \nmid \ell N_1 N_2 \) differ by the \(
p^r \). Let $N = \ell N_1 N_2 $. For $p$ co-prime to $N$, we have $a_p
(f_1 ) =  a_p (f_2 ) p^r $. Since $k_1 = k_2+2r$, 
\begin{equation*}
L_N (f_1, s) =  \prod_{p \nmid N } \left( 1 - \frac{a_p (f_2)}{
    p^{(s-r)} } 
+  \frac{p^{k_2-1}}{p^{2(s-r)}} \right)^{-1}  
            =  L_N(f_2, s - r ) =  
\frac{ L ( f_2 , s - r )}{ \prod_{ p \mid N }  E_p ( f_2, s - r)  }
\end{equation*} 
Thus, we get:
\begin{equation} \label{eq:Lf1-Lf2}
L (f_1 , s ) = \left(   \prod_{ p \mid N }  
\frac{ E_p ( f_1, s)}{ E_p ( f_2, s-r)} \right) \cdot L (f_2 , s-r ) .
\end{equation}
Let us recall the functional equations satisfied by 
\( L (f_1, s) \) and \( L (f_2, s-r ) \):
\begin{equation} \label{eq:FEf1}
A_{f_1} (s ) \Gamma (s ) L (f_1, s) 
= \lambda (f_1) A_{f_1} ( k_1 -s ) \Gamma ( k_1 -s ) L (f_1, k_1 -s) .
\end{equation}
\begin{equation} \label{eq:FEf2}
A_{f_2} (s-r ) \Gamma (s -r) L (f_2, s-r) 
= \lambda (f_2) A_{f_2} ( k_2 +r -s ) \Gamma ( k_2 +r -s ) L (f_2, k_2 +r -s) .
\end{equation}

Dividing (\ref{eq:FEf1}) by (\ref{eq:FEf2}) we get:

\begin{equation} \label{eq:5}
\begin{split}
\frac{A_{f_1} (s )}{A_{f_2} (s-r )} &\cdot 
\frac{ \Gamma (s)}{\Gamma (s -r)}  \cdot 
\prod_{p \mid N} \frac{ E_p ( f_1, s)}{ E_p ( f_2, s-r)}\\
&= 
\frac{\lambda (f_1) }{\lambda (f_2) } \cdot 
\frac{A_{f_1} ( k_1 - s )}{A_{f_2} ( k_2 + r - s )} \cdot
\frac{\Gamma (k_1 -s )}{\Gamma ( k_2 + r -s )} \cdot 
\frac{ L (f_1 , k_1 - s) }{L(f_2, k_2 +r - s)} 
\end{split}
\end{equation}
Using (\ref{eq:Lf1-Lf2}) with \( s \) replaced by \( k_1 - s \) 
and the \( \Gamma \)-function identity, 
\( \frac{ \Gamma (s)}{  \Gamma ( s - m)} = 
 \prod_{j=1}^m ( s - j) \)  for \( m = r \), 
in (\ref{eq:5}), we get:
\begin{equation} \label{eq:6}
\begin{split}
&\frac{A_{f_1} (s )}{A_{f_2} (s-r )} \cdot 
\prod_{j=1}^r ( s - j) \cdot 
\prod_{p \mid N} \frac{ E_p ( f_1, s)}{ E_p ( f_2, s-r)}\\
&=
\frac{\lambda (f_1) }{\lambda (f_2) } \cdot 
\frac{A_{f_1} ( k_1 - s )}{A_{f_2} ( k_2 + r - s )} \cdot
( - 1 )^r \prod_{j=1}^r ( s - (k_1 - j) ) \cdot 
\frac{ E_p ( f_1, k_1 - s)}{ E_p ( f_2, k_2 + r - s )}  
\end{split}
\end{equation}
We re-write the above equation as
\begin{equation}\label{eq:7}
\frac{ \prod_{j=1}^r ( s - j) }{\prod_{j=1}^r ( s - (k_1 - j) )} 
= \lambda \cdot 
A(s)  \cdot 
\prod_{p \mid N} \frac{ E_p(f_1, k_1 - s) E_p(f_2, s-r)}
{ E_p ( f_2, k_2 + r - s ) E_p ( f_1, s) }  
\end{equation}
The right hand side of (\ref{eq:7}) can be written as a product 
of a nowhere zero entire function and a quotient of finite Euler products: 
\[
\lambda \cdot A(s) \cdot \frac{\prod_{i=1}^a(1-\alpha_i {p_i}^{-s})}
{\prod_{j=1}^b(1-\beta_j {p_i}^{-s})}
\]
where $a, ~b$ are natural numbers, and $\alpha_i, ~1\leq i\leq  a$ and 
$\beta_j, ~1\leq j\leq b$ are some
complex numbers and \( p_i \) are prime numbers dividing \( N \). Here
$A(s)$ is a nowhere vanishing entire function, and $\lambda$ is a
non-zero constant. 

We now claim that \( r = 0 \). Let us assume the contrary. 

Note that there is no cancellation of terms between the numerator 
and the denominator on the left hand side. For otherwise, for some 
\( 1 \leq j_1 , j_2 \leq r \), we have \( k_1 = j_1 + j_2 \leq 2r \), 
which is a contradiction as \( k_1 - 2r = k_2 > 0 \). 
 
This implies that the right hand side has to have a zero, in turn
forcing the existance of a term of the form \( (1 - \alpha_i
{p_i}^{-s} )\)  for  some \( i \) and some non-zero \( \alpha_i
\). Note that any such term has an infinitely many zeros 
\( s = \frac{\ln \alpha_i +  2 \pi  i n}{\ln p_i } \) 
for \( n \in \Z \). This is a
contradiction as there are only a finitely many zeroes to the left
hand side. 

Thus, to sum up, examining the zeros of the two sides of 
Equation (\ref{eq:7}), we conclude that 
$r=0$ and \( k_1 = k_2 \), and this proves the proposition.
\end{proof}

\begin{remark} A different way of looking at this result is to use the
  fact that the Galois representation attached to a modular form of
  weight $k$, arises from a motive with Hodge weights  $(0,k-1),
  (k-1,0)$. Twisting by the cyclotomic character $\omega^r$
  (equivalently by a Tate twist $-r$-times), the Hodge weights are of
  the form $(r, k-1+r), (k-1+r, r)$. Thus, if the twisted representation has
  to arise from a newform, then $r=0$.   
\end{remark}

\begin{remark} Another way of seeing the validity of Proposition
  \ref{mf} is to consider the automorphic
  representations attached to the cusp forms $f_1$ and $f_2\otimes \chi$. By the
  strong multiplicity one theorem of Piatetskii-Shapiro and
  Jacquet-Shalika, 
the corresponding automorphic representations are isomorphic. The
weight of the modular form is given by the discrete series, the
component at infinity of the corresponding automorphic
representation. Hence it follows that $r=0$. 
\end{remark}

\begin{remark} The above proof allows one to prove
analogues of Theorem \ref{MP}, in the context of
  Hilbert modular forms, or even more generally whenever one can
  attach Galois representations to automorphic forms or even in a
  geometric context for the $l$-adic representations arising from
  cohomology of varieties defined over global fields with fixed Hodge
  structures. 
\end{remark}

\section*{Acknowledgements}
The first author thanks the School of Mathematics, Tata Institute of
Fundamental Research, Mumbai for its excellent hospitality and work
environment. We thank Dipendra Prasad for some interesting
discussions.

\end{document}